\documentclass[a4paper,oneside]{amsart}
\usepackage[foot]{amsaddr}
\usepackage[utf8]{inputenc}
\usepackage[T1]{fontenc}
\usepackage[english]{babel}
\usepackage[textwidth = 14cm, textheight = 22cm]{geometry}
\usepackage[backend=bibtex,style=alphabetic,url=false,giveninits=true]{biblatex} 
\usepackage{amsmath,amssymb,amsthm}
\usepackage{mathtools, bm, esint}
\usepackage{csquotes}
\usepackage[colorlinks=true]{hyperref}
\usepackage{algorithm,algpseudocode}
\usepackage{caption,subcaption}
\patchcmd{\subsubsection}{\itshape}{\bfseries}{}{}

\newcommand{\diff}{\mathop{}\mathopen{}\mathrm{d}}
\newcommand{\prob}{\mathrm{Prob}}

\newcommand{\Q}{\mathbb{Q}}
\newcommand{\R}{\mathbb{R}}

\newcommand{\N}{\mathbb{N}}
\newcommand{\indic}{\mathbf{1}}
\newcommand{\deb}{\rightharpoonup}
\newcommand{\mres}{\mathop{\hbox{\vrule height 6pt width .5pt depth 0pt \vrule height .5pt width 4pt depth 0pt}}\nolimits}
\newcommand{\eps}{\varepsilon}
\DeclareMathOperator*{\supp}{supp}
\DeclareMathOperator*{\diam}{diam}
\DeclareMathOperator{\hdm}{\mathcal{H}}
\DeclareMathOperator{\lbm}{\mathcal{L}}
\DeclareMathOperator{\prox}{prox}
\DeclareMathOperator*{\argmin}{arg min}
\DeclarePairedDelimiter{\abs}{\lvert}{\rvert}
\DeclarePairedDelimiter{\norm}{\lVert}{\rVert}

\numberwithin{equation}{section}

\theoremstyle{plain}
\newtheorem{thm}{Theorem}[section]
\newtheorem{lem}[thm]{Lemma}
\newtheorem{prop}[thm]{Proposition}

\newtheorem{conj}[thm]{Conjecture}

\theoremstyle{definition}
\newtheorem{defn}{Definition}[section]
\newtheorem*{pjgm}{The projected gradient method}
\newtheorem*{pxgm}{The proximal gradient method}
\newtheorem*{tpgm}{A \enquote{twisted} proximal gradient method}
\newtheorem*{plbfgs}{The projected L-BFGS method}

\theoremstyle{remark}
\newtheorem{rem}{Remark}[section]

\begin{document}

\title{A fractal shape optimization problem in branched transport}
\author{Paul Pegon$^\dagger$}
\email[bli]{paul.pegon@math.u-psud.fr}
\author{Filippo Santambrogio$^\dagger$}
\email{filippo.santambrogio@math.u-psud.fr}
\address[$\dagger$]{Laboratoire de Math\'ematiques d'Orsay, Univ. Paris-Sud, CNRS, Universit\'e Paris-Saclay, 91405 Orsay cedex, France}
\author{Qinglan Xia$^\ddagger$}
\email{qlxia@math.ucdavis.edu}
\address[$\ddagger$]{Department of Mathematics, UC Davis, One Shields Ave, Davis, CA 95616, United States}

\begin{abstract}
We investigate the following question: what is the set of unit volume which can be best irrigated starting from a single source at the origin, in the sense of branched transport? We may formulate this question as a shape optimization problem and prove existence of solutions, which can be considered as a sort of \enquote{unit ball} for branched transport. We establish some elementary properties of optimizers and describe these optimal sets $A$ as sublevel sets of a so-called landscape function which is now classical in branched transport. We prove $\beta$-H\"older regularity of the landscape function, allowing us to get an upper bound on the Minkowski dimension of the boundary: $\overline{\dim}_M \partial A \leq d-\beta$ (where $\beta \coloneqq d(\alpha-(1-1/d))\in (0,1)$ is a relevant exponent in branched transport, associated with the exponent $\alpha>1-1/d$ appearing in the cost). We are not able to prove the lower bound, but we conjecture that $\partial A$ is of non-integer dimension $d-\beta$. Finally, we make an attempt to compute numerically an optimal shape, using an adaptation of the phase-field approximation of branched transport introduced some years ago by Oudet and the second author.
\end{abstract}


\maketitle

\tableofcontents

\section*{Introduction}

Given two probability measures $\mu,\nu$ on $\R^d$, a classical optimization problem amounts to finding a connection between the two measures which has minimal cost. In branched transport, such a connection will be performed along a $1$-dimensional structure such that the cost for moving a mass $m$ at distance $\ell$ is proportional to $m^\alpha \times \ell$ where $\alpha$ is some concave exponent $\alpha \in [0,1]$. The map $t \mapsto t^\alpha$ being subadditive (even strictly subadditive for $\alpha < 1$), that is to say $(a+b)^\alpha \leq a^\alpha + b^\alpha$, it is cheaper for masses to travel together as much as possible. Consequently, the optimal connections exhibit branching structures: for instance, if one wishes to transport one Dirac mass to two Dirac masses of mass $1/2$, the optimal graph will be $Y$-shaped.

A early model has been proposed by Gilbert in \cite{Gil} as an extension of Steiner problem (see \cite{GilPol}) in a discrete setting, where the connection between two atomic measures is made through weighted oriented graphs. There are two main extensions of this model to a continuous setting, i.e. with arbitrary probability measures. The first one was introduced in 2003 by the third author in \cite{Xia03} and can be viewed as a Eulerian model. It is based on vector measures and roughly reads as:
\[\min \quad \left\{ \int \abs*{\frac{dv}{\diff\!\hdm^1}(x)}^\alpha \diff\hdm^1(x) : \nabla \cdot v  = \mu - \nu \right\},\]
minimizing among vector measures which have an $\hdm^1$-density. A Lagrangian model was introduced essentially at the same time by and Maddalena, Solimini, Morel \cite{MadSolMor}, and then intensively studied by Bernot, Caselles, Morel \cite{BerCasMor05} . It is based on measures on a set of curves, but the description of this model, which is a little more involved, is given in Section \ref{sec:prelim}. An almost up-to-date reference on branched transport resides in the book by the same authors \cite{BerCasMor}.

Looking at the optimal branching structures computed numerically in \cite{OudSan} (in some non-atomic cases), or at natural drainage networks and their irrigations basins, one is tempted to describe them as fractal (see \cite{RodRin}). Actually, even though the underlying network has infinitely many branching points, it is stil a $1$-rectifiable set, hence it is not clear in what sense fractality appears. Fractality is a notion which usually relates either to self-similarity properties of non-smooth objects, or to non-integer dimension of sets. A first rigorous result which would fall in the first category is proven by Brancolini and Solimini in \cite{BraSol14}: for sufficiently diffuse measures (for example the Lebesgue measure restricted to a Lipschitz open set), the number of branches of length $\sim \eps$ stemming from a branch of length $\ell$ is of order $\ell/\eps$. This may read as a self-similarity property since in a way the total length is preserved when looking at subbranches at all scales.

The present paper leans towards the other notion of fractality, that is towards \enquote{fractal} dimension. Some sets in branched transport have already been proposed as candidates to exhibit non-integer dimension, for instance the boundary of adjacent irrigation basins (an open conjecture by J.-M. Morel). Here we are interested in another candidate which is related to branched transport: the boundary of what we call unit balls for branched transport. With the results of the present paper, we can only prove an upper bound on the dimension, which is non-integer, and conjecture that this upper bound is actually sharp.

The article is divided into five parts. In a preliminary section we define properly the Lagrangian framework of branched transport and its basic features, and we formulate our question as a shape optimization problem involving the irrigation distance. Section \ref{sec:existence} is devoted to the proof of existence of minimizers and to elementary properties of minimizers. In Section \ref{sec:hölder} we prove the $\beta$-Hölder regularity of the landscape function, which appears in the description of optimizers, and use it to derive an upper bound on the Minkowski dimension of the boundary of the optimizers in Section \ref{sec:4}. The final section is an attempt at computing optimizers numerically, which is particularly useful due to the fact that we are not fully able to answer theoretically the question of the fractal behavior of the boundary. This is done by adapting the Modica-Mortola approach introduced by \cite{OudSan}, and allows to provide some convincing computer visualizations.

\section{Preliminaries}\label{sec:prelim}

As preliminaries, we quickly set the Lagrangian framework of branched transport and its main features. For more details, we refer to the book \cite{BerCasMor} or to \cite[Sections 1--2]{Peg} for a simpler exposition.

\subsection{The irrigation problem}

We denote by $\Gamma(\R^d)$ the set of $1$-Lipschitz curves in $\R^d$ parameterized on $[0,\infty]$, endowed with the topology of uniform convergence on compact sets.

\subsubsection*{Irrigation plans}  We call \emph{irrigation plan} any probability measure $\eta \in \prob({\Gamma})$ satisfying the following finite-length condition
\begin{equation}\label{fin_len}
\mathbf{L}(\eta) \coloneqq \int_\Gamma L(\gamma) \:\diff \eta(\gamma) < +\infty,
\end{equation}
where $L(\gamma) = \int_0^\infty \abs{\dot{\gamma}(t)} \diff t$. Notice that any irrigation plan is concentrated on $\Gamma^1(\R^d) \coloneqq \{ \gamma : L(\gamma) < \infty\}$. We denote by $\mathrm{IP}(\R^d)$ the set of all irrigation plans $\eta\in \prob({\Gamma})$. If $\mu$ and $\nu$ are two probability measures on $\R^d$, one says that $\eta \in \mathrm{IP}(\R^d)$ irrigates $\nu$ from $\mu$ if one recovers the measures $\mu$ and $\nu$ by sending the mass of each curve respectively to its initial point and to its final point, which means that
\[
(\pi_0)_\#\eta = \mu \text{ and }
(\pi_\infty)_\#\eta = \nu,
\]
where $\pi_0(\gamma) = \gamma(0)$, $\pi_\infty(\gamma) = \gamma(\infty) \coloneqq \lim_{t \to +\infty} \gamma(t)$ and $f_\#\eta$ denotes the push-forward of $\eta$ by $f$ whenever $f$ is a Borel map\footnote{Notice that $\lim_{t\to\infty} \gamma(t)$ exists if $\gamma \in \Gamma^1(K)$, and this is all we need since any irrigation plan is concentrated on $\Gamma^1(K)$.}. We denote by $\mathrm{IP}(\mu,\nu)$ the set of irrigation plans irrigating $\nu$ from $\mu$:
\[\mathrm{IP}(\mu,\nu) = \{ \eta \in \mathrm{IP}(\R^d) : (\pi_0)_\#\eta = \mu, (\pi_\infty)_\#\eta = \nu\}.\]
If $\eta$ is a given irrigation plan, we define the multiplicity at $x$, that is the total mass passing by $x$, as
\[\theta_\eta(x) = \eta(\{\gamma \in \Gamma : x \in \gamma\}),\]
where $x\in \gamma$ means that $x$ belongs to the image of the curve $\gamma$. Finally, for any nonnegative function $f$, we denote by $\int_\gamma f(x) \diff x$ the line integral of $f$ along $\gamma \in \Gamma$:
\[\int_\gamma f(x) \diff x \coloneqq \int_0^{+\infty} f(\gamma(t)) \abs{\dot{\gamma}(t)} \diff t.\]

\subsubsection*{Irrigation costs}
For $\alpha \in [0,1]$ we consider the \emph{irrigation cost} $\mathbf{I}_\alpha : \mathrm{IP}(\R^d) \to [0,\infty]$ defined by
\[\mathbf{I}_\alpha(\eta) \coloneqq \int_\Gamma \int_\gamma \theta_\eta(x)^{\alpha -1} \diff x\diff\eta(\gamma),\]
with the conventions $0^{\alpha -1} = \infty$ if $\alpha <1$, $0^{\alpha -1} = 1$ otherwise, and $\infty \times 0 = 0$. If $\mu,\nu$ are two probability measures on $\R^d$, the irrigation (or branched transport) problem consists in minimizing the cost $\mathbf{I}_\alpha$ on the set of irrigation plans which send $\mu$ to $\nu$, which reads
\begin{equation}\label{lag_pb}\tag{$\text{LI}_\alpha$}
\min_{\eta \in \mathrm{IP}(\mu,\nu)}\quad \int_\Gamma \int_\gamma \theta_\eta(x)^{\alpha -1} \diff x \diff\eta(\gamma).
\end{equation}
We set $Z_\eta(\gamma) = \int_\gamma \theta_\eta(x)^{\alpha-1} \diff x$ so that the cost may expressed as
\[\mathbf{I}_\alpha(\eta) = \int_\Gamma Z_\eta(\gamma) \diff\eta(\gamma).\]

The following results are extracted from \cite{BerCasMor05,Peg,Xia03}.

\begin{prop}[First variation inequality for $\mathbf{I}_\alpha$]
If $\eta$ is an irrigation plan with $\mathbf{I}_\alpha(\eta)$ finite, then for all irrigation plan $\tilde{\eta}$ the following holds:
\begin{equation}
\mathbf{I}_\alpha(\tilde{\eta}) \leq \mathbf{I}_\alpha(\eta) + \alpha \int Z_\eta(\gamma) \diff(\tilde{\eta}-\eta).
\end{equation}
Notice that the integral $\int Z_\eta \diff(\tilde{\eta}-\eta)$ is well-defined since $\int Z_\eta \diff\eta < \infty$ and $Z_\eta$ is nonnegative, though it may be infinite.
\end{prop}

\begin{thm}[Existence of minimizers,]
For any pair of probability measures $\mu,\nu \in \prob(\R^d)$ which have compact support, the problem \eqref{lag_pb} admits a minimizer.
\end{thm}

\begin{thm}[Irrigability]
If $1- \frac{1}{d} < \alpha < 1$, for any $\mu,\nu \in \prob(\R^d)$ with compact support there exists some $\eta \in \mathrm{IP}(\mu,\nu)$ such that $\mathbf{I}_\alpha(\eta)$ is finite.
\end{thm}

From now on we assume that $\alpha \in \left]1-\frac{1}{d}, 1 \right[$.

\subsubsection*{Irrigation distance}
Let us set
\[d_\alpha(\mu,\nu) = \min \{\mathbf{I}_\alpha(\eta) \; : \; \eta \in \mathrm{IP}(\mu,\nu)\}\]
for any pair $\mu,\nu$ of probability measures on $\R^d$. For any compact $K \subseteq \R^d$, it induces a distance on $\prob(K)$ which metrizes the weak-$\star$ convergence of measures in the duality with $\mathcal{C}(K)$. On non-compact subsets of $\R^d$, the distance $d_\alpha$ is lower semicontinuous w.r.t. the  weak-$\star$ convergence of measures in the duality with bounded and continuous functions (narrow convergence)\footnote{Proving this is just an adaptation of the proof on compact sets. If $\mu$ is fixed (for example) and $\nu_n \to \nu$ with $\eta_n \in \mathrm{IP}(\mu,\nu_n)$ optimal and parameterized by arc length, assuming that the cost is bounded, the irrigation plans $\eta_n$ are tight and one may extract a subsequence converging to some $\eta$ which irrigates $\nu$ and whose cost is less than $\liminf d_\alpha(\mu,\nu_n)$ by lower semicontinuity of $\mathbf{I}_\alpha$.}.

\begin{prop}[Scaling law]\label{prp:bt_law}
For any compactly supported measures $\mu,\nu$ with equal mass, there is an upper bound on the irrigation distance depending on the mass and the diameter. We set $\mu' = \mu - \mu\wedge \nu, \nu' = \nu - \mu \wedge \nu$ the disjoint parts of the measures and $m = \abs{\mu'} = \abs{\nu'}$ their common mass. Then:
\[d_\alpha(\mu,\nu) \leq C m^\alpha \diam(\supp \mu' \cup \supp \nu').\]
\end{prop}
\subsubsection*{Landscape function}
The landscape function was introduced by the second author in \cite{San}, in the single-source case. It has been then studied by Brancolini, Solimini in \cite{BraSol11} and by the third author in \cite{Xia14}. It will be a central tool in the study of the shape optimization problem we are going to introduce. We recall here the basic definitions and properties. Given an optimal irrigation plan $\eta \in \mathrm{IP}(\delta_0,\nu)$, we say that a curve $\gamma$ is $\eta$-good if
\begin{itemize}
\item the quantity $Z_\eta(\gamma) = \int_\gamma \theta_\eta(x)^{\alpha-1}  \diff x$ is finite,
\item for all $t < T(\gamma)$,
\[\theta_\eta(\gamma(t)) = \eta(\{\tilde{\gamma}\in \Gamma(\mathbb{R}^d) : \gamma = \tilde{\gamma} \text{ on } [0,t]\}),\]
where $T(\gamma)=\inf \{t\in [0,\infty]: \gamma(s)=\gamma(\infty) \text{ for all } s\in[ t, \infty]\}$ is the stopping time of $\gamma$.
\end{itemize}

One may prove by optimality that $\eta$ is concentrated on the set of $\eta$-good curves. Moreover it is proven in \cite{San} that for all $\eta$-good curves $\gamma$, the quantity $Z_\eta(\gamma)$ depends only on the final point $\gamma(\infty)$ of the curve, thus we may define the landscape function $z_\eta$ as follows:
\[z_\eta(x) = \begin{dcases*}
Z_\eta(\gamma)&
if $\gamma$ is an $\eta$-good curve s.t. $x = \gamma(\infty)$,\\
+\infty&
otherwise.
\end{dcases*}\]
Notice that for an optimal $\eta$ the cost may be expressed in terms of $z_\eta$:
\[\mathbf{I}_\alpha(\eta) = \int_\Gamma Z_\eta(\gamma) \diff\eta(\gamma) = \int_{\R^d} z_\eta(x) \diff\nu(x).\]
Finally, one may show that $z_\eta$ is lower semicontinuous and that the inequality $z_\eta(x)\geq |x|$ holds.

\subsection{The shape optimization problem}

We ask ourselves the following question: what is the set of unit volume which is closest to the origin in the sense of irrigation? To give this a precise meaning, we embed everything in the space of probability measures; hence we want to minimize the $d_\alpha$ distance between the unit Dirac mass at $0\in \mathbb{R}^d$ and sets $E$ of unit volume, seen as the uniform measure on $E$. This problem reads
\begin{equation}\tag{$\text{S}_\alpha$}\label{pb:S}
\min\quad \{d_\alpha(\delta_0, \indic_E \lbm) : \abs{E} = 1\},
\end{equation}
where $\lbm$ denotes the Lebesgue measure on $\mathbb{R}^d$. We relax this problem by minimizing on a larger set, which is the set of probability measures with Lebesgue density bounded by $1$, thus getting:
\begin{equation}\tag{$\text{R}_\alpha$}\label{pb:R}
\min\quad \{\mathbf{X}_\alpha(\nu) : \nu \leq 1, \nu \in \prob(\R^d)\},
\end{equation}
where $\mathbf{X}_\alpha(\nu) = d_\alpha(\delta_0, \nu)$.

In the following, we will sometimes encounter positive measures which do not have unit mass, thus we extend the functional by setting $\mathbf{X}_\alpha(\nu) \coloneqq d_\alpha(\abs{\nu}\delta_0, \nu)$ for any finite measure $\nu$.

A key tool in the analysis of this problem lies in the following proposition, proved in \cite{San} under slightly more restrictive hypotheses.
\begin{prop}[First variation inequality for $\mathbf{X}_\alpha$]\label{prop:first_var}
Suppose that $\nu\in \mathrm{Prob}(\mathbb{R}^d)$ with $\mathbf{X}_\alpha(\nu)< \infty$. 
Suppose also that $\eta$ is an optimal irrigation plan between $\abs{\nu}\delta_0$ and $\nu$, with landscape function $z_\eta$. The following holds:
\[\mathbf{X}_\alpha(\tilde{\nu}) \leq \mathbf{X}_\alpha(\nu) + \alpha \int z_\eta \diff(\tilde{\nu}-\nu)\]
for any $\tilde{\nu}\in \mathrm{Prob}(\mathbb{R}^d)$.
\end{prop}
Notice also that the integral $\int z_\eta \diff(\tilde{\nu}-\nu)$ is well-defined since $\int z_\eta \diff\nu=\mathbf{I}_{\alpha}(\eta)=\mathbf{X}_{\alpha}(\nu) < \infty$ and $z_\eta$ is non-negative, though it may be infinite.
\begin{proof}
If $\int z_\eta \diff \tilde{\nu} = \infty$ then there is nothing to prove. Otherwise for $\nu$-a.e. $x$, $z_\eta(x)$ is finite hence there are $\eta$-good curves reaching $x$ and one can find a measurable\footnote{One can characterize $\eta$-good curves as those $\gamma$ such that $\tilde{Z}_\eta(\gamma) < \infty$ where $\tilde{Z}_\eta(\gamma) \coloneqq \int_0^\infty \abs{\gamma}_{t,\eta} \,\diff t$ is a slight variation of $Z_\eta$ defined in \cite{San} which is also lower semicontinuous. Hence the multifunction associating to every $x$ the set of $\eta$-good curves reaching $x$ can be written as $\bigcup_{\ell \in \Q}\{\gamma\in\Gamma\,:\,\tilde{Z}_\eta(\gamma)\leq \ell, \gamma(\infty)=x\}$,  i.e. as a countable union of multifunctions with closed graph. This means that this multifunction is measurable and admits a measurable selection (see e.g. \cite{CasVal}).} map $g:\mathbb{R}^d\rightarrow \Gamma$ which associates with every $x$ an $\eta$-good curve reaching $x$. Let us build an irrigation plan $\tilde{\eta} \in \mathrm{IP}(\abs{\tilde{\nu}}\delta_0, \tilde{\nu})$ which is concentrated on $\eta$-good curves, by setting $\tilde{\eta} = g_\# \nu$, so that
\[\int_\Gamma Z_\eta \diff\tilde{\eta} = \int_\Gamma z_\eta(\gamma(\infty)) \diff\tilde{\eta}(\gamma) = \int_{\R^d} z_\eta(x) \diff\tilde{\nu}.\]
Then, by the first variation inequality for $\mathbf{I}_\alpha$, we get:
\[\mathbf{X}_\alpha(\tilde{\nu}) \doteq d_\alpha(\abs{\tilde{\nu}}\delta_0,\tilde{\nu}) \leq \mathbf{I}_\alpha(\tilde{\eta}) \leq \mathbf{I}_\alpha(\eta) + \alpha \int_\Gamma Z_\eta \diff(\tilde{\eta}-\eta) = \mathbf{X}_\alpha(\nu) + \alpha \int_{\R^d} z_\eta \diff (\tilde{\nu}-\nu).\qedhere\] 
\end{proof}

\section{Existence and first properties}\label{sec:existence}

We will often denote by $C = C(\alpha,d)$ or $c = c(d)$ different positive constants which depend only on $\alpha,d$ or $d$ respectively.

\begin{thm}
The relaxed shape optimization problem \eqref{pb:R} admits at least a minimizer.
\end{thm}

\begin{proof}
The existence of a minimizer follows from the lower semicontinuity and tightness. Indeed, any minimizing sequence $\nu_n$ must have bounded first moment since 
$$\int |x|\diff \nu(x)\leq \int z_\eta(x)\diff \nu(x)=d_\alpha(\delta_0,\nu).$$
A bound on the first moment implies tightness of the sequence and, up to extracting a subsequence, one has $\nu_n\deb\nu$. The condition $\nu_n\leq 1$ implies $\nu\leq 1$ and the lower semicontinuity of $d_\alpha$ provides the optimality of $\nu$.
\end{proof}
For $1>\alpha>1-\frac{1}{d}$, we will denote the optimal value for the relaxed shape optimization problem $\eqref{pb:R}$ by:
\[e_{\alpha}: = \min\{d_{\alpha}(\delta_0,\nu): \nu \le 1 \text{ and } \nu \in \prob(\R^d) \}.\]

\begin{lem}[Scaling lemma]\label{lem:scaling}
For any finite measure $\nu$ we have
\[\mathbf{X}_{\alpha}(\nu) \geq e_{\alpha} \abs{\nu}^{\alpha+\frac{1}{d}}.\]
\end{lem}
\begin{proof}
For $\lambda=\abs{\nu}^{-1/d}$, let $\tilde{\nu}=\lambda^d \varphi_{\#}(\nu)$ be a scaling of $\nu$ under the map $\varphi(x)=\lambda x$ in $\R^d$. Then, 
$\int_{\R^d}\diff\tilde{\nu}=\lambda^d \int_{\R^d} \diff\nu=\lambda^d \abs{\nu}=1$ and $\nu \leq 1$. Thus,

\[e_{\alpha}\leq d_{\alpha}(\tilde{\nu},  \delta_0)=\lambda^{\alpha d+1} d_{\alpha}(\nu, \abs{\nu} \delta_0)=\abs{\nu}^{-\left(\alpha+\frac{1}{d}\right)} \mathbf{X}_\alpha(\nu).\]
\end{proof}

For any $\nu$, we say that $z$ is a landscape function of $\nu$ if it is the landscape function $z_\eta$ associated with some optimal irrigation plan $\eta \in \mathrm{IP}(\delta_0,\nu)$.

\begin{thm}\label{thm:exist}
Let $\nu$ be a minimizer of \eqref{pb:R} and $z$ a landscape function of $\nu$. Then $\nu$ is the indicator of a set $A$ which is a sublevel set of $z$:
\begin{equation}
\label{eq: z*_value}
A = \{x : z(x) \leq z^\star \}, \text{ with }z^\star = \frac{e_{\alpha}}{\alpha}\left(\alpha+\frac{1}{d}\right).
\end{equation}
In particular, $A$ is a solution to problem \eqref{pb:S} and it is a compact and path-connected set.
\end{thm}
\begin{proof}
We show that $\nu$ also minimizes the first variation of $\mathbf{X}_\alpha$, that is $\mu \mapsto \int z \diff \mu$. Take $\tilde{\nu}$ a competitor for \eqref{pb:R}. By Proposition \ref{prop:first_var}, one has:
\[\mathbf{X}_\alpha(\tilde{\nu}) \leq \mathbf{X}_\alpha(\nu) + \alpha \int z \diff(\tilde{\nu}-\nu),\]
but $\mathbf{X}_\alpha(\nu) \leq \mathbf{X}_\alpha(\tilde{\nu})$, thus
\[\int z \diff\nu \leq \int z \diff\tilde{\nu}\]
for any $\tilde{\nu}$. So as to minimize this quantity, $\nu$ must concentrate its mass on the points where $z$ takes its lowest values. More precisely, there is a value $z^\star \in [0,\infty]$ such that
\[\nu(x) \begin{dcases*} = 1& if $z(x) < z^\star$,\\
\in [0,1]& if $z(x) = z^\star$,\\
= 0& if $z(x) > z^\star$.
\end{dcases*}\]
Indeed, we just take $z^\star = \sup \{t \in \R : \abs{\{z(x) \leq t\}} < 1\}$. Since $\int z \diff \nu = e_\alpha > 0$, necessarily $z^\star > 0$.
This kind of arguments is typical in optimization problems under an upper density constraint, as it was for instance done for crowd motion applications in \cite{MauRouSan}.

\paragraph{\textbf{Step 1:} $z^\star \leq \frac{e_\alpha}{\alpha} \left(\alpha + \frac{1}{d}\right)$}~

For $0 \leq k < z^\star$, we consider the competitor $\tilde{\nu} = \indic_{\{z \leq k\}}$ and set $\abs{\tilde{\nu}} = 1 - m$, noting that $m > 0$ by definition of $z^\star$. Using Lemma \ref{lem:scaling} and Proposition \ref{prop:first_var}, one gets
\begin{align*}
e_\alpha (1-m)^{\alpha+\frac{1}{d}} \leq \mathbf{X}_\alpha(\tilde{\nu}) \leq  \mathbf{X}_\alpha(\nu)+\alpha \int z \diff(\tilde{\nu}-\nu) = e_\alpha-\alpha \int_{\{z>k\}}z \diff\nu.
\end{align*}
Since $\nu(\{z > k\}) = 1 - \abs{\{z \leq k\}} = m$, it follows that
\begin{equation}
\label{eq: step1eqn}
e_\alpha (1-m)^{\alpha+\frac{1}{d}} \leq e_\alpha - \alpha k m.
\end{equation}
As $\alpha + \frac{1}{d} > 1$, the map $t \mapsto t^{\alpha + \frac{1}{d}}$ is (strictly) convex, thus
\[e_\alpha \left( 1 - \left(\alpha + \frac{1}{d}\right) m \right) \leq e_\alpha (1-m)^{\alpha+\frac{1}{d}} \leq e_\alpha - \alpha k m,\]
hence forgetting the middle term, substracting $e_\alpha$ and dividing by $m$:
\[
\alpha k \leq e_\alpha \left(\alpha + \frac{1}{d}\right).
\]
Taking the limit $k \to z^\star$ yields:
\begin{equation}\label{eq:zstar_leq}
z^\star \leq \frac{e_\alpha}{\alpha} \left(\alpha + \frac{1}{d}\right).
\end{equation}

\paragraph{\textbf{Step 2:} $\nu = \indic_A$ where $A = \{z\leq z^\star\}$}~

Take the competitor $\tilde{\nu} = \indic_{\{z \leq z^\star\}}$ and set $\abs{\tilde{\nu}} = 1 + m$, $m \geq 0$. Using again the scaling lemma and the first variation of $\mathbf{X}_\alpha$ one gets:
\[e_\alpha (1+m)^{\alpha + \frac{1}{d}} \leq e_\alpha + \alpha\int_{z = z^\star} z \diff (\tilde{\nu}-\nu) = e_\alpha + \alpha z^\star m.\]
Now by strict convexity of $t \mapsto t^{\alpha + \frac{1}{d}}$, if $m > 0$ then one has
\begin{gather*}
 e_\alpha (1+m)^{\alpha + \frac{1}{d}} > e_\alpha \left(1 + \left(\alpha+\frac{1}{d}\right)m\right),
 \shortintertext{thus}
e_\alpha \left(\alpha+ \frac{1}{d}\right)m < \alpha z^\star m,
\end{gather*}
which contradicts \eqref{eq:zstar_leq}. Consequently $m = 0$, hence $\nu = \tilde{\nu} = \indic_{\{z \leq z^\star\}}$.

\paragraph{\textbf{Step 3:} Compactness and connectedness}~

$A$ is closed since $z$ is lower semicontinuous and bounded since $z(x) \geq \abs{x}$ for all $x\in \R^d$. It is path-connected since any point $x$ with $z(x) \leq z^\star$ is the endpoint of an $\eta$-good curve $\gamma$ starting from $0$ and $\gamma \subseteq A$ because $z$ is increasing along this curve.

\paragraph{\textbf{Step 4:} $z^\star \geq \frac{e_\alpha}{\alpha}\left(\alpha +\frac{1}{d}\right)$}~

Take $x_0 \in A$ with maximal Euclidean norm. Then the half ball $H_r(x_0) \coloneqq B_r(x_0) \cap \{ x : \langle x - x_0 , x_0 \rangle > 0\}$ is included in $A^c$. We consider the competitor $\tilde{\nu} = \indic_{A \sqcup H_r(x_0)}$, with mass $\abs{\tilde{\nu}} = 1 + m$, where $m = \abs{H_r(x_0)} = c r^d$ for some constant $c = c(d)$. To irrigate $\tilde{\nu}$, we pay at most the cost of irrigation of $\nu$, plus the price for moving an extra mass $m$ from $0$ to $x_0$ along the irrigation plan, plus the cost for moving this mass to $B_r(x_0) \setminus A$, which we can bound by $C m^\alpha r$ thanks to Proposition \ref{prp:bt_law}, as follows:
\begin{align*}
\mathbf{X}_\alpha(\tilde{\nu}) = d_\alpha((1+m)\delta_0, \tilde{\nu}) &\leq d_\alpha((1+m)\delta_0, \nu+ m\delta_{x_0}) + d_\alpha(\nu+ m\delta_{x_0}, \nu + \indic_{H_r(x_0)})\\
&= \mathbf{X}_\alpha(\nu + m\delta_{x_0}) + d_\alpha(m\delta_{x_0},\indic_{H_r(x_0)})\\
&\leq e_\alpha + \alpha m z(x_0) +  C r m^\alpha,
\end{align*}
where $C = C(\alpha,d)$ is some positive constant. Combining this inequality with the following convexity inequality
\[\mathbf{X}_\alpha(\tilde{\nu})\ge e_\alpha (1+m)^{\alpha + \frac{1}{d}} \geq e_\alpha \left( 1 + \left(\alpha +\frac{1}{d}\right)m\right),\] and dividing by $m >0$, one gets:
\[e_\alpha\left(\alpha + \frac{1}{d}\right) \leq \alpha z(x_0) + C r^{1+d\alpha -d}.\]
Passing to the limit $r\to 0$, we obtain
\[z^\star \geq z(x_0) \geq \frac{e_\alpha}{\alpha}\left(\alpha +\frac{1}{d}\right).\qedhere\]
\end{proof}

\section{Hölder continuity of the landscape function}\label{sec:hölder}

The Hölder regularity of the landscape function has been proved in \cite{San} under some regularity assumptions on $\nu$ using Campanato spaces (these spaces were introduced in \cite{Cam}, see \cite[Section 2.3]{Giu} for a modern exposition). Namely, if $\nu$ is of the form $\nu = f \lbm_{\mres E}$ where the density $f(x)$ and the fraction of mass $\Theta_E(x,r) \coloneqq \frac{\abs*{E \cap B_r(x)}}{\abs*{B_r(x)}}$ lying in $E$ are bounded from below by some constant $c > 0$ for all $x\in E$ and all $r \leq \diam E$, then $z$ is $\beta$-Hölder continuous where
\[\beta \coloneqq d \left(\alpha - \left(1- \frac{1}{d}\right)\right) = 1 + d \alpha - d,\]
is a number which is strictly between $0$ and $1$ as $1 > \alpha > 1 - \frac{1}{d}$. Another proof for more general regularity assumptions on $\nu$ has been given in \cite{BraSol11}. In our case, we do not know a priori that $A$ is regular (on the contrary we suspect it has a fractal boundary), hence the Hölder regularity of $z$ does not follow from previous works. Exploiting the fact that $A$ is optimal, we are going to show that $z$ is $\beta$-Hölder continuous adapting classical computations to pass from Campanato to Hölder spaces. More precisely, setting $A_r(x) \coloneqq A \cap B_r(x)$ and $z_r(x)$ the mean of $z$ on $A_r(x)$, we are going to prove the following sequence of inequalities, for arbitrary $r \leq \diam A$:
\begin{align}
\fint_{A_r(x)} \abs{z-z_r(x)} &\leq C r^\beta,\\
\abs{z_r(x)-z_{r/2}(x)} &\leq C r^\beta,\\
z_r(x) - z(x) &\leq C r^\beta,\\
\abs{z(x) - z_r(x)} &\leq C r^\beta,\\
\abs{z_{\abs{y-x}}(x) - z_{\abs{y-x}}(y)} &\leq C \abs{y-x}^\beta.
\end{align}
Notice that the two last inequalities imply that $z$ is indeed $\beta$-Hölder continuous:
\begin{align*}
\abs{z(y)-z(x)} &\leq \abs{z(y) - z_{\abs{y-x}}(y)} + \abs{z_{\abs{y-x}}(x) - z_{\abs{y-x}}(y)} + \abs{z(x) - z_{\abs{y-x}}(x)}\\
&\leq C \abs{y-x}^\beta.
\end{align*}
The main difficulty we will encounter is that we will quite easily obtain estimates of the form
\[ \cdots \leq C \frac{r^\beta}{\Theta_A(x,r)^{1-\alpha}},\]
and will need to get rid of the term $\Theta_A(x,r)^{1-\alpha}$, i.e. treat the case when it becomes small.

\subsection{Main lemmas}

The following lemma will be key to prove the regularity of the landscape function.

\begin{lem}[Maximum deviation]\label{lem:dev_max}
There is a constant $C = C(d,\alpha) > 0$ such that the following holds:
\begin{equation}\label{eq:dev_max}
\forall y\in A_r(x), \qquad z^\star - z(y) \leq C \frac{r^\beta}{\Theta_{A^c}(x,r)^{1-\alpha}}.
\end{equation}
\end{lem}

\begin{proof}
We consider the competitor $\tilde{\nu} = \indic_{A \cup B_r(x)}$, with mass $\abs{\tilde{\nu}} = 1 + m$ where $m = \abs{B_r(x) \setminus A}$. For any $y\in A_r(x)$, let us irrigate $\tilde{\nu}$ from $0$ by irrigating $\nu$ from $0$, moving an extra mass $m$ from $0$ to $y$ along the irrigation plan, then irrigating $\indic_{B_r(x)\setminus A}$ from this mass at $y$. Using Lemma  \ref{lem:scaling} and Proposition \ref{prp:bt_law}, we have
\[e_\alpha(1+m)^{\alpha + \frac{1}{d}} \leq \mathbf{X}_\alpha(\tilde{\nu}) \leq \mathbf{X}_\alpha(\nu) + \alpha \int z \diff (m\delta_y) + C r m^\alpha.\]
By convexity,
\[e_\alpha\left(1 + \left(\alpha+\frac{1}{d}\right)m\right) \leq (1+m)^{\alpha + \frac{1}{d}} e_\alpha \leq e_\alpha + \alpha m z(y) + C r m^\alpha,\]
thus, knowing that $e_\alpha \left(\alpha + \frac{1}{d}\right) = \alpha z^\star$ by \eqref{eq: z*_value}:
\[\alpha m z^\star \leq \alpha m z(y) + C r m^\alpha.\]
By definition, $m = \omega_d r^d \Theta_{A^c}(x,r)$ where $\omega_d$ is the volume on the unit $d$-dimensional ball, hence
\[ z^\star - z(y) \leq C r (r^d \Theta_{A^c}(x,r))^{\alpha -1} = C \frac{r^\beta}{\Theta_{A^c}(x,r)^{1-\alpha}}.\qedhere\]
\end{proof}

\begin{rem}
One can see that if $\Theta_A(x,r)$ becomes small, then $\Theta_{A^c}(x,r)$ is large (close to $1$), and actually all values of $z$ in $A_r(x)$ become close to the same value $z^\star$ up to $C r^\beta$.
\end{rem}

\begin{lem}[Mean deviation]\label{lem:mean_dev}
There is some constant $C = C(d,\alpha) > 0$ such that
\[ \fint_{A_r(x)} \abs{z(y) - z_r(x)} \diff y \leq C r^\beta\]
for all $r >0$ and all $x \in A$.
\end{lem}
\begin{proof}
We will first show that
\[ \fint_{A_r(x)} \abs*{z(y) - z_r(x)} \diff y \leq C \frac{r^\beta}{\Theta_A(x,r)^{1-\alpha}}.\]
Denoting by $\bar{z}_r(x)$ the central median of $z$ on the set $A_r(x)$, there is a disjoint union $A_r(x) = A^- \sqcup A^+$ such that $\abs{A^-} = \abs{A^+} = \frac{\abs{A_r(x)}}{2}$ and $z \leq \bar{z}_r(x)$ on $A^-$, $z \geq \bar{z}_r(x)$ on $A^+$.
Let us consider the competitor $\tilde{\nu} = \indic_A - \indic_{A^+} + \indic_{A^-}$. By the first variation lemma:
\[\mathbf{X}_\alpha(\tilde{\nu}) \leq \mathbf{X}_\alpha(\nu) + \alpha \int z \diff(\tilde{\nu}-\nu).\]
Recall that $\mathbf{X}_\alpha(\rho) = d_\alpha(\delta_0,\rho)$ when $\rho$ is a probability measure, which is the case for $\nu$ and $\tilde{\nu}$, and that $d_\alpha$ is a distance. Thus by the triangle inequality:
\[\alpha \int z \diff (\nu - \tilde{\nu}) \leq d_\alpha(\delta_0,\nu) - d_\alpha(\delta_0,\tilde{\nu}) \leq d_\alpha(\nu,\tilde{\nu}).\]
We know that $d_\alpha(\nu,\tilde{\nu}) \leq C \abs{\tilde{\nu}-\nu}^\alpha \diam(\supp(\tilde{\nu}-\nu)) \leq C \abs{A_r(x)}^\alpha r$ for some $C = C(\alpha,d) > 0$. Moreover notice that
\begin{align*}
\int z \diff(\nu-\tilde{\nu}) &= \int_{A^+} z(y) \diff y - \int_{A^-} z(y) \diff y\\
&= \int_{A^+} (z(y)-\bar{z}_r(x)) \diff y + \int_{A^-} (\bar{z}_r(x)-z(y)) \diff y\\
 &= \int_{A_r(x)} \abs{z(y) - \bar{z}_r(x)} \diff y.
 \end{align*}
Consequently:
\[\fint_{A_r(x)} \abs{z(y) - \bar{z}_r(x)} \diff y \leq C \abs{A_r(x)}^{\alpha -1} r \leq C \frac{\abs{A_r(x)}^{\alpha -1}}{\abs{B_r(x)}^{\alpha -1}} r^{1+d(\alpha -1)} = C \frac{r^\beta}{\Theta_A(x,r)^{1-\alpha}}.\]
Moreover, one has
\[\abs{z_r(x) - \bar{z}_r(x)} = \abs*{\fint_{A_r(x)} z(y) - \bar{z}_r(x)\diff y}  \leq \fint_{A_r(x)} \abs{z(y) - \bar{z}_r(x)}\diff y \leq C \frac{r^\beta}{\Theta_A(x,r)^{1-\alpha}}\]
which leads to
\[ \fint_{A_r(x)} \abs{z(y) - z_r(x)} \diff y = \fint_{A_r(x)} \abs{z(y) - \bar{z}_r(x)}\diff y + \abs{z_r(x) - \bar{z}_r(x)} \leq C \frac{r^\beta}{\Theta_A(x,r)^{1-\alpha}}.\]
Now we get rid of $\Theta_A(x,r)^{1-\alpha}$. If $\Theta_A(x,r) \geq 1/2$, we get the desired inequality. On the other hand, if $\Theta_{A^c}(x,r) \geq 1/2$, by Lemma \ref{lem:dev_max}, we have
\[0 \leq z^\star - z(y) \leq C r^\beta, \qquad \forall y \in A_r(x),\]
which also implies that
\[0 \leq z^\star - z_r(x) \leq C r^\beta.\]
By these two inequalities, we have
\[\abs{z(y)-z_r(x)} \leq C r^\beta, \qquad \forall y \in A_r(x).\]
Now, taking the mean over $A_r(x) \ni y$ leads to the wanted inequality as well:
\[ \fint_{A_r(x)} \abs{z(y) - z_r(x)} \diff y \leq C r^\beta.\qedhere\]
\end{proof}
\begin{rem}
Notice that the estimate
\[ \fint_{A_r(x)} \abs{z(y) - z_r(x)} \diff y \leq C \frac{r^\beta}{\Theta_A(x,r)^{1-\alpha}}\]
is valid in general: we only use the fact that $\nu$ is an indicator function (a density bounded from below would suffice). The optimality of $\nu$ comes into play to to get rid of $\Theta_A(x,r)$.
\end{rem}

\subsection{Hölder regularity}

\begin{prop}[Small-scale difference]\label{prop: small_scale}
For all $x \in A$ and all $r > 0$  one has
\[\abs{z_r(x)-z_{r/2}(x)} \leq C r^\beta.\]
\end{prop}
\begin{proof}
First we show that
\[\abs{z_r(x)-z_{r/2}(x)} \leq C \frac{r^\beta}{\Theta_A(x,r)^{1-\alpha}}.\]
Indeed, by Lemma \ref{lem:mean_dev},
\begin{align*}
\abs{z_r(x)-z_{r/2}(x)} &\leq \frac{\int_{A_{r/2}(x)} \abs{z(y)-z_r(x)}\diff y}{\abs{A_{r/2}(x)}} \leq \frac{\abs{A_r(x)}}{\abs{A_{r/2}(x)}} \fint_{A_r(x)} \abs{z(y)-z_{r}(x)}\diff y\\
&\leq 2^d \frac{\Theta_A(x,r)}{\Theta_A(x,r/2)} C r^\beta \leq C \frac{r^\beta}{\Theta_A(x,r/2)}.
\end{align*}
As before, if $\Theta_A(x,r/2) \geq 1/2$ we get the desired estimate. Otherwise $\Theta_{A^c}(x,r/2) \geq 1/2$ and $\Theta_{A^c}(x,r) \geq 2^{-d} \Theta_{A^c}(x,r/2) \geq 2^{-1-d}$. Now, by Lemma \ref{lem:dev_max},
\[0 \le z^\star - z(y) \leq C \frac{r^\beta}{\Theta_{A^c}(x,r)} \leq C r^\beta, \qquad \forall y \in A_r(x).\]
Consequently $0 \le z^\star - z_{r/2}(x) \leq C r^\beta$ and $0 \le z^\star - z_r(x) \leq C r^\beta$ which implies that
\[\abs{z_r(x) - z_{r/2}(x)} \leq C r^\beta.\qedhere\]
\end{proof}

\begin{lem}[Lower deviation to the mean]
There is a constant $C = C(d,\alpha) > 0$ such that for all $x\in A$ and all $r>0$ one has:
\begin{equation}\label{low_dev_mean}
\forall y \in A_r(x), \qquad z_r(x) - z(y) \leq C r^\beta.
\end{equation}
\end{lem}
\begin{proof}
First we show that
\[z_r(x) - z(y) \leq C \frac{r^\beta}{\Theta_A(x,r)^{1-\alpha}}.\]
Remove the mass $m=\abs{A_r(x)}$ going to $A_r(x)$ from the irrigation plan, make it travel along the plan to any fixed $y \in A_r(x)$ and then send it to $A_r(x)$: this should cost more. This implies
\[\alpha m z(y) - \alpha \int_{A_r(x)} z + C m^\alpha r \geq 0,\]
which may be rewritten as
\[z_r(x) - z(y) \leq C m^{\alpha-1} r \leq C \frac{r^\beta}{\Theta_A(x,r)^{1-\alpha}}.\]
Now if $\Theta_A(x,r) \geq 1/2$ one gets the desired result. Otherwise $\Theta_{A^c}(x,r) \geq 1/2$ and Lemma \ref{lem:dev_max} yields:
\[0 \le z^\star - z(y) \leq C r^\beta, \qquad \forall y \in A_r(x).\]
Thus $0\le z^\star - z_r(x)\leq C r^\beta$ and for any fixed $y \in A_r(x)$, 
\[\abs{z_r(x) - z(y)} \leq \abs{z_r(x) - z^\star}+\abs{z^\star - z(y)} \leq C r^\beta,\]
from which we also get the wanted inequality.
\end{proof}

\begin{lem}[Deviation to the mean]\label{lem:dev_mean}
For all $x\in A$ and all $r > 0$, one has
\[\abs{z(x)-z_r(x)} \leq C r^\beta.\]
\end{lem}
\begin{proof}
By Proposition \ref{prop: small_scale}, one has
\[\abs{z(x) - z_r(x)} \leq \abs{z(x) - z_{r/2}(x)} + \abs{z_{r/2}(x) - z_r(x)} \leq \abs{z(x)-z_{r/2}(x)} + C r^\beta,\]
which means by setting $f(r) = \abs{z(x) - z_r(x)}$ for $r>0$ that:
\[f(r) \leq f(r/2) + C r^\beta.\]
Consequently for all $k \in \N$
\[f(r) \leq f(r \cdot 2^{-(k+1)}) + Cr^\beta \sum_{i=0}^k 2^{-i\beta}\]
thus
\[f(r) \leq \limsup_{\eps\to 0} f(\eps) + Cr^\beta \sum_{i=0}^\infty 2^{-i\beta} \leq \limsup_{\eps\to 0} f(\eps) + Cr^\beta.\]
Now let us prove that $f(\eps) \to 0$ when $\eps\to 0$, i.e. $z_\eps(x) \xrightarrow{\eps\to 0} z(x)$. We already know that $z$ is lower semi-continuous hence $z(x) \leq \liminf_{\eps\to 0} z_{\eps}(x)$. Moreover using \eqref{low_dev_mean}, we have
\[\limsup_{\eps\to 0} z_\eps(x) \leq \limsup_{\eps\to 0} (z(x) + C\eps^\beta) = z(x),\]
which implies that $z_\eps(x) \to z(x)$ when $\eps\to 0$. Therefore the inequality $f(r) \leq C r^\beta$ holds, that is to say:
\[ \abs{z(x) - z_r(x)} \leq C r^\beta.\qedhere\]
\end{proof}

\begin{lem}[Large scale difference]\label{lem:large_scale}
For any $x,y \in A$, one has:
\[\abs{z_{\abs{y-x}}(x)-z_{\abs{y-x}}(y)} \leq C \abs{y-x}^\beta.\]
\end{lem}
\begin{proof}
Set $r = \abs{y-x}$, and $\Delta_r = B_r(x) \cap B_r(y)$.  Notice that, $\Delta_r$ being a fixed fraction of $B_r(x)$ (independant of $r$), $\abs{\Delta_r} = c \abs{B_r}$ for some $c = c(d) \in (0,1)$. 

If both $\Theta_{A^c}(x,r) \ge \frac{c}{2}$ and $\Theta_{A^c(}y,r) \ge \frac{c}{2}$, then by Lemma \ref{lem:dev_max} one has:
\begin{align*}
0\le z^\star - z_r(x) \leq C \frac{r^\beta}{\Theta_{A^c}(x,r)^{1-\alpha}} \leq C r^\beta,\text{ and }
0  \le z^\star - z_r(y) \leq C \frac{r^\beta}{\Theta_{A^c}(y,r)^{1-\alpha}} \leq C r^\beta,
\end{align*}
which implies the desired inequality
\begin{equation}\label{eq: z_rxy}
\abs{z_r(x)-z_r(y)} \leq C r^\beta.
\end{equation}

On the other hand, if either $\Theta_{A^c}(x,r)$ or $\Theta_{A^c}(y,r)$ is less than $c/2$, say $\Theta_{A^c}(x,r) \le \frac{c}{2}$, we claim the desired inequality (\ref{eq: z_rxy}) still holds. Indeed, for all $u\in A_r(x) \cap A_r(y)$ one has
\[\abs{z_r(x)-z_r(y)} \leq \abs{z_r(x) - z(u)} + \abs{z_r(y) - z(u)}\]
thus integrating over $A_r(x) \cap A_r(y)$ in $u$ one gets:
\begin{align*}
\abs{z_r(x)-z_r(y)}  &\leq \frac{1}{\abs{A_r(x)\cap A_r(y)}}\left[\int_{A_r(x)} \abs{z(u)-z_r(x)} \diff u +\int_{\abs{A_r(y)}} \abs{z(u)-z_r(y)} \diff u\right]\\
&\leq C r^\beta \frac{\abs{A_r(x)}+\abs{A_r(y)}}{\abs{A_r(x)\cap A_r(y)}},
\end{align*}
the last inequality resulting from Lemma \ref{lem:mean_dev}. Note that
\[\abs{A_r(x) \cap A_r(y)} = \abs{\Delta_r \cap A} \geq \abs{\Delta_r} - \abs{B_r(x) \setminus A} = c \abs{B_r(x)} - \abs{B_r(x) \setminus A}\]
which implies that
\[\frac{\abs{A_r(x)}+\abs{A_r(y)}}{\abs{A_r(x) \cap A_r(y)}} \leq \frac{2 \abs{B_r(x)}}{c \abs{B_r(x)} - \abs{B_r(x) \setminus A}}
=\frac{2}{c-\Theta_{A^c}(x,r)}\le \frac{4}{c}.\]
Thus, in this case, we still have
\[\abs{z_r(x)-z_r(y)} \leq C r^\beta\frac{\abs{A_r(x)}+\abs{A_r(y)}}{\abs{A_r(x) \cap A_r(y)}}\leq C r^\beta.\]
\end{proof}

\begin{thm}[Hölder continuity]\label{thm: Holder}
The function $z$ is $\beta$-Hölder continuous on $A$. More precisely:
\[\forall x,y \in A, \quad \abs{z(y)-z(x)} \leq C \abs{y-x}^\beta,\]
for some constant $C = C(\alpha,d)$.
\end{thm}
\begin{proof}
By Lemma \ref{lem:dev_mean} and Lemma \ref{lem:large_scale},
\[\abs{z(y)-z(x)} \leq \abs{z(y) - z_{\abs{y-x}}(y)} + \abs{z_{\abs{y-x}}(y)-z_{\abs{y-x}}(x)} + \abs{z(y)-z_{\abs{y-x}}(y)} \leq 3 C \abs{y-x}^\beta.\qedhere\]
\end{proof}

As a consequence of this result we may quantify the minimal size of a ball one can put inside $A$ around $x$ in terms of $z^\star - z(x)$ and prove that $A$ has non-empty interior.

\begin{prop}[Interior points]\label{prop:interior_points}
For some constant $C = C(\alpha,d)$ the following holds:
\begin{equation}\label{eq:ball_inclusion}
\forall x\in A, \qquad B_{r(x)}(x)\subseteq A,
\end{equation}
where $r(x) \coloneqq C(z^\star-z(x))^{1/\beta}\ge 0$. 
In particular \[\{x\in A: z(x) < z^\star\} \subseteq \overset{\circ}{A} \text{ and }\partial A \subseteq \{x\in A: z(x) = z^\star\}.\]
\end{prop}
\begin{proof}
It suffices to prove \eqref{eq:ball_inclusion} for $x_0\in A$ satisfying $z(x_0) < z^\star$. Consider a point $x \in A^c$. Take a point $y \in A$ which is closest to $x$ : it is possible since $A$ is compact. By Lemma \ref{lem:dev_max}, we know that for small $r$
\[\Theta_{A^c}(y,r)^{1-\alpha} (z^\star - z(y))\leq C r^\beta.\]
But by construction $y$ is such that $\liminf_{r\to 0} \Theta_{A^c}(y,r) \geq 1/2$, and since the right-hand side tends to $0$, necessarily $z(y) = z^\star$. By the Hölder continuity of $z$ stated in Theorem \ref{thm: Holder},
\begin{align*}
z^\star - z(x_0) = \abs{z(y) - z(x_0)} \leq C \abs{y-x_0}^\beta \leq C \abs{x-x_0}^\beta,
\end{align*}
where the last inequality follows from the fact that $\abs{y-x_0} \leq \abs{y-x} + \abs{x-x_0} \leq 2 \abs{x-x_0}$ because $y$ minimizes the distance from $x$. Hence, for all $x\in A^c$, $\abs{x-x_0} \ge C(z^\star-z(x_0))^{1/\beta}=r(x_0)$, which implies the desired result.
\end{proof}

\section{On the dimension of the boundary}\label{sec:4}

We are interested in the dimension of the boundary $\partial A$, our guess being that it should be non-integer, and lie between $d-1$ and $d$. Here we look at the Minkowski dimension (also called box-counting dimension). Given a set $X$, we denote by $N_\eps(X)$ the maximum amount of disjoint balls of radius $\eps$ centered at points of $X$.

\begin{defn}[Minkowski dimension]
We define the upper Minkowski dimension of $X$ by
\[\overline{\dim}_M(X) = \limsup_{\eps \to 0} \frac{\log(N_\eps(X))}{-\log \eps},\]
and the lower Minkowski dimension by
\[\underline{\dim}_M(X) = \liminf_{\eps \to 0} \frac{\log(N_\eps(X))}{-\log \eps}.\]
When these coincide we just call it the Minkowski dimension and denote it by $\dim_M(X)$.
\end{defn}

We shall get an upper bound on the upper Minkowski dimension. We say that $X$ is of dimension smaller than $\delta$ if $\overline{\dim}_M X \leq \delta$.

\begin{lem}\label{lem:volume_bound}
There is a constant $C=C(\alpha,d)$ such that for all $k \leq z^\star$,
\[\abs{\{x\in A: k< z(x) \leq z^\star\}} \leq C (z^\star -k).\]
\end{lem}
\begin{proof}
Consider the competitor $\tilde{\nu} = \indic_{\{z \leq k\}}$ with total mass $\abs{\tilde{\nu}} = 1 - m$, where $m=|\{x\in A: k<z(x) \le z^*\}|$. As in (\ref{eq: step1eqn}), one has
\[e_\alpha (1-m)^{\alpha + 1/d} \leq e_\alpha - \alpha km \]
hence knowing that $\alpha z^\star = (\alpha + 1/d) e_\alpha$ and developing the term on the left-hand side at order $2$, we obtain:
\[-\alpha m z^\star + \frac{e_{\alpha}}{2}(\alpha + \frac{1}{d})(\alpha+\frac{1}{d}-1)m^2 \leq -\alpha km\]
Thus
\[m \leq C(z^\star - k)\]
with $1/C = e_{\alpha}(\alpha + 1/d)(\alpha+1/d-1)/(2\alpha)$.
\end{proof}

\begin{thm}
The set $\partial A$ is of dimension less than $d-\beta$.
\end{thm}
\begin{proof}
For $\eps > 0$ fixed, take disjoint balls $(B_i)_{i \in I}$ of radius $\eps$, where $N \coloneqq \abs{I} = N_\eps(\partial A)$. We set $B_i^+ = B_i \setminus A$, $B_i^- = B_i \cap A$. We split the set of balls into two parts: those which have a larger intersection with $A$ rather than $A^c$, and vice-versa. Namely, we set
\begin{align*}
I^+ &= \{i \in I: \abs{B_i^+} \geq \abs{B_i}/2\},&
N^+ &= \abs{I^+},\\
I^- &= \{i \in I: \abs{B_i^-} \geq \abs{B_i}/2\},&
N^- &= \abs{I^-},
\end{align*}
so that $I = I^+ \cup I^-$ and $N \leq N^+ + N^-$. We are going to bound $N^+$ and $N^-$ by some power of $\eps$.

\paragraph{\textbf{Step 1:} Bound on $N^-$}~

Since $z$ is $\beta$-Hölder continuous on $A$, one has for each $B_i = B_\eps(x_i)$:
\[\forall x \in B_i \cap A,\qquad \abs{z(x) - z^\star} < C \eps^\beta,\]
since the center $x_i$ lies in $\partial A \subseteq \{z = z^\star\}$ according to Proposition \ref{prop:interior_points}. Consequently
\[(\partial A)_\eps \cap A \subseteq \{ z^\star - C \eps^\beta < z \leq z^\star\},\]
thus because of Lemma \ref{lem:volume_bound}:
\[\abs{(\partial A)_\eps \cap A} \leq \abs{\{ z^\star - C \eps^\beta < z \leq z^\star\}} \leq C \eps^\beta.\]
Using the previous inequality and the fact that $\abs{B_i^-} \geq \abs{B_i}/2 \geq C \eps^d$ for $i\in I^-$, one has:
\[C N^- \eps^d \leq \sum_{i \in I^-} \abs{B_i^-} \leq \abs{(\partial A)_\eps \cap A} \leq C \eps^\beta,\]
which implies:
\begin{equation}\label{eq:Nminus}
N^- \leq C \eps^{-(d-\beta)}.
\end{equation}

\paragraph{\textbf{Step 2:} Bound on $N^+$}~

We consider the competitor $\tilde{\nu} = \indic_{\tilde{A}}$ where $\tilde{A} = A \cup \bigcup_{i\in I^+} B_i^+$. It has a mass $\abs{\tilde{\nu}} = 1 + m$ where $m = \sum_{i\in I^+} \abs{B_i^+}$. To irrigate $\tilde{\nu}$, we send an extra mass $\abs{B_i^+}$ to each center $x_i$ along the irrigation plan, which costs $\alpha\abs{B_i^+} z^\star$, then we send this mass towards $B_i^+$, which costs at most $C\abs{B_i^+}^\alpha \eps$. But one should get a cost no less than $e_\alpha(1+m)^{\alpha + 1/d}$ by the scaling lemma. Moreover, with a development of order $2$ one has:
\begin{align*}
(1+m)^{\alpha + 1/d} &\geq 1 + (\alpha + 1/d)m + 1/2 \cdot (\alpha + 1/d)(\alpha + 1/d-1) (1+m)^{\alpha +1/d -2} m^2\\
&\geq 1 + (\alpha + 1/d)m + C  m^2
\end{align*}
because for $\eps$ small, $1+m$ is less than $2$ for example. Consequently one may say:
\[e_\alpha \left(1 + (\alpha + 1/d)m + C m^2 \right) \leq e_\alpha + \alpha m z^\star + \sum_{i\in I^+} C\eps \abs{B_i^+}^\alpha.\]
Recall that $\alpha z^\star = e_\alpha (\alpha +1/d)$, thus after simplifying one gets for some $C > 0$:
\begin{equation}\label{eq:step_Nplus}
m^2 \leq C \sum_{i\in I^+} \abs{B_i^+}^\alpha \eps \leq C N^+ \eps^{1+\alpha d}.
\end{equation}
Notice that for $i \in I^+$, $\abs{B_i^+} \geq \abs{B_i}/2 \geq C \eps^d$, so that
\[m = \sum_{i\in I^+} \abs{B_i^+} \geq C N^+ \eps^d.\]
Injecting this into \eqref{eq:step_Nplus}, one gets:
\[(N^+\eps^d)^2 \leq C N^+ \eps^{1+\alpha d},\]
thus
\begin{equation}\label{eq:Nplus}
N^+ \leq C \eps^{1+\alpha d-2d} = C \eps^{-(d-\beta)}.
\end{equation}
Putting \eqref{eq:Nminus} and \eqref{eq:Nplus} together yields:
\[N_\eps(\partial A) = N \leq N^+ + N^- \leq \frac{C}{\eps^{d-\beta}},\]
and
\[\overline{\dim}_M(\partial A) = \limsup_{\eps \to 0} \frac{\log(N_\eps(\partial A))}{-\log(\eps)} \leq d-\beta,\]
which means that $\partial A$ is of dimension smaller than $d-\beta$.
\end{proof}
This result pushes us to propose the following conjecture:

\begin{conj}
The boundary $\partial A$ is of dimension $d-\beta$, in the sense that:
\[\dim_H(\partial A) = \dim_M(\partial A) = d-\beta.\]
\end{conj}
Proving this requires to establish the inequality $\dim_H(\partial A) \geq d-\beta$, for which we do not have a working strategy yet.

\section{Numerical simulations}

Our goal now is to compute solutions to our shape optimization problem numerically. To perform numerical simulations, we use the Eulerian framework of branched transport, first defined by the third author in \cite{Xia03}. This framework is based on vector measures with a measure divergence, i.e. measures $v \in \mathcal{M}^d(\R^d)$ such that $\nabla \cdot v \in \mathcal{M}(\R^d)$, the set of such measures being denoted by $\mathcal{M}_{div}(\R^d)$. The cost is the so-called $\alpha$-mass:
\[M_\alpha(v) = \begin{dcases*}
\int  \abs*{\frac{dv}{\diff\!\hdm^1}(x)}^\alpha \diff\hdm^1(x)& if $v$ is $1$-rectifiable,\\
+\infty& otherwise.
\end{dcases*}\]
An elliptic approximation of this functional was introduced by Oudet and the second author in \cite{OudSan} (see also \cite{SanCRAS}), in the spirit of Modica and Mortola \cite{ModMor}. The approximate functional is defined for $\eps > 0$ by:
\[M^\alpha_\eps(v) = \eps^{-\sigma_1} \int \abs{v(x)}^\sigma \diff x + \eps^{\sigma_2} \int \frac{\abs{v(x)}^2}{2} \diff x\]
for suitably chosen $\sigma,\sigma_1,\sigma_2$. It is proven in \cite{OudSan} that $M_\eps^\alpha$ $\Gamma$-converges to $M^\alpha$ as $\eps$ goes to $0$, for a suitable topology on $\mathcal{M}_{div}(\R^d)$. Moreover, the $\Gamma$-convergence result also holds imposing an equality constraint on the divergence $\nabla \cdot v = f_\eps$, for a suitable sequence $f_\eps\deb f$, as proven in \cite{Mon17}. The results of \cite{OudSan} are proven in dimension $d=2$, but in \cite{Mon15} there is a proof of how to extend to higher dimension, in the case $\alpha>1-1/d$ (in dimension $d=2$ there is also a version of the $\Gamma$-convergence result for $\alpha \leq 1/2$). Also note that, recently, other phase-field approximations for branched transport or other network problems have been studied, see for instance \cite{BonOrlOud,ChaFerMer,BonLemSan}.

Here we adapt the approach of \cite{OudSan} to our shape optimization problem by adding this time an inequality constraint on the divergence.

Recall that the Lagrangian and Eulerian frameworks are equivalent \cite{Peg}, so that the irrigation distance may be computed in the following way:
\[d_\alpha(\mu,\nu) = \inf_v \quad \{ M^\alpha(v) \quad : \quad \nabla \cdot v = \mu - \nu\}.\]
Consequently the shape optimization problem \eqref{pb:R} rewrites, in relaxed form, as:
\begin{equation}\tag{ES}\label{pb:ES}
\min_v \quad \{ M^\alpha(v) \; :\; \mu -1 \leq \nabla \cdot v \leq \mu \} \quad \text{where $\mu=\delta_0$}.
\end{equation}
Setting $a = \mu-1, b = \mu$ and some mollified versions $a_\eps = \mu_\eps-1, b_\eps = \mu_\eps$,for example a convolution of $\mu$ with the standard mollifier of suitable size $r_\eps$ (e.g. $\eps^{\sigma_2} r_\eps^{-d} = o(1)$ as in \cite{Mon17}), we define the following approximate problem, for $\eps > 0$:
\begin{equation}\label{pb:AS}\tag{AS}
\min_v \quad \{ M^\alpha_\eps(v) \; :\; a_\eps \leq \nabla \cdot v \leq b_\eps \}.
\end{equation}

Let us remark that the above-mentioned $\Gamma$-convergence results do not allow us to say that this problem approximates \eqref{pb:ES}, as the inequality constraint on the divergence is not directly in these works. We leave this question for further investigation, as our aim is for now to make a first attempt to compute numerically an optimal shape for the original problem \eqref{pb:S}.

\subsection{Optimization methods}

We tackle problem \eqref{pb:AS} by descent methods. Two difficulties arise: first of all, the functional $M_\eps^\alpha$ is not convex hence there is no garantee that the methods converge, and if they do, they may converge to a local minimizer which is not necessarily a global minimizer ; secondly, this is a constrained problem, hence we will need to compute projections or resort to proximal methods to handle the constraint. The simplest approach is to use a first-order method, for instance to perform a projected gradient descent on the functional $M^\alpha_\eps$ for $\eps$ fixed (but small):
\begin{pjgm}
\[\left|\begin{aligned}
v_0 &\in C\\
v_{n+1} &= p_C(v_n - \tau_n \nabla M^\alpha_\eps(v_n)),
\end{aligned}\right.\]
where
\[C = \{ v : a_\eps \leq \nabla \cdot v  \leq b_\eps\}\]
is the convex set of admissible vector fields for \eqref{pb:AS}.
\end{pjgm}

Computing the projection $p_C$ is not an easy task, even more so as we want fast computations since this projection should be done at each step of the algorithm. Actually, this projection step will be quite costly (at least in our approach), hence we need to pass to a higher order method to get to an approximate minimizer in a reasonable number of iterations.

Recall that the projected gradient method is a particular case of the proximal gradient method, which we describe briefly. Consider a problem of the form
\[\min_v f(v) + g(v)\]
where $f$ is smooth and $g$ \enquote{proximable}, in the sense that one may easily compute its proximal operator
\[\prox^\tau_g(v) = \argmin_{v'} g(v') + \frac{1}{2\tau} \abs{v'-v}^2.\]
The proximal gradient method consists in doing at each step an explicit descent for $f$ and an implicit descent for $g$:
\begin{pxgm}
\[\left|\begin{aligned}
v_0 &\text{ given}\\
v_{n+1} &= \prox_g^{\tau_n}(v_n - \tau_n \nabla f(v_n)).
\end{aligned}\right.\]
\end{pxgm}
The projected gradient method corresponds to the case
\[g(v) = \begin{dcases*}
0&
if $v \in C$,\\
+\infty&
otherwise.
\end{dcases*}\]
If there was no function $g$, we recover the classical gradient descent method. Notice that there is an implicit choice in this method, since we compute gradients which depend on the scalar product. There is no reason that the canonical scalar product is well adapted to the function we want to minimize. Following the work of Lee, Sun and Saunders \cite{LeeSunSau} on Newton-type proximal methods, one may \enquote{twist} the scalar product, leading to the more general method: 
\begin{tpgm}
\begin{equation}\label{met:gen_prox}
\left|\begin{aligned}
v_0 &\text{ given}\\
v_{n+1} &= \prox_g^{\tau_n,H_n}(v_n - \tau_n \nabla_{H_n} f(v_n)),
\end{aligned}\right.
\end{equation}
where $\nabla_H f(x)$ is the gradient of $f$ with respect to the scalar product $\langle x,y\rangle_H = \langle H x, y \rangle$ for $H$ an invertible self-adjoint operator, and
\[\prox^{\tau,H}_g(v) = \argmin_{v'} g(v') + \frac{1}{2\tau} \norm{v'-v}_H^2.\]
\end{tpgm}
The best quadratic model of $f$  around a point $x_0$ is
\begin{align*}
Qf_{x_0}(x) &= f(x_0) + \langle \nabla_H f(x_0) , x \rangle_H + 1/2 \langle x,x \rangle_H,
\end{align*}
with $H = H_f(x_0)$ being the Hessian of $f$ at $x$, thus it is natural to consider \eqref{met:gen_prox} with $H_n = H_f(x_n)$. Notice indeed that if $g$ is zero, the proximal operator is the identity and that $\nabla_H f(v) = H^{-1} \nabla f$, so that one recovers Newton's method:
\[v_{n+1} = v_n - \tau_n H_n^{-1} \nabla f(v_n),\]
which is known to converge quadratically for smooth enough $f$. This is why this method is called \emph{proximal Newton method}. However, for large-scale problems, computing and storing the Hessian is very costly, thus an alternative is to set $H_n$ to be an approximation of the Hessian of $f$ at $v_n$, thus leading to \emph{proximal quasi-Newton methods}. These methods were introduced in \cite{LeeSunSau}, which we refer to for further detail and theoretical results of convergence.

A very popular choice for $H_n$ is given by the L-BFGS method (see \cite{LiuNoc}), which is a quasi-Newton method building in some sense the \enquote{best} approximation of the Hessian at $v_n$ using only the information of the points $v_k$ and the gradients $\nabla f(v_k)$ for a fixed number of previous steps $k = n, n-1, \ldots, n-L+1$. The interest is that no matrix is stored, and there is a very efficient way to compute the matrix-vector product $H_n^{-1} \cdot v$ using simple algebra. Therefore, we decided to implement a proximal L-BFGS method, which in our case reads:
\begin{plbfgs}
\begin{equation}
\left|\begin{aligned}
v_0 &\text{ given}\\
v_{n+1} &= p_C^{\tilde{H}_n}(v_n - \tau_n \tilde{H}_n^{-1}\nabla f(v_n)),
\end{aligned}\right.
\end{equation}
where $\tilde{H}_n$ is the approximate Hessian computed with the L-BFGS method with $L$ steps and $p_C^{\tilde{H}_n}$ is the projection on $C$ with respect to the norm $\norm{\cdot}_{\tilde{H}_n}$.
\end{plbfgs}
The algorithm to compute the matrix-vector product $\tilde{H}_n^{-1} \cdot x$ is given in Section \ref{sec:algo}.

\subsection{Computing the projection}\label{sec:proj}

The difficulty lies in the computation of the projection, that is on the proximal operator. A box constraint on the variable is very easy to deal with, but here we are faced with box constraints on $\nabla \cdot v$, that is on a linear operator applied to $v$. Moreover, we want to compute a projection with respect to a twisted scalar product $\langle \cdot, \cdot \rangle_H$, which adds some extra difficulty. For simplicity of notations, we rename $a_\eps, b_\eps$ as $a,b$. By definition, finding the projection $p_C^H(v_0)$ of $v_0$ amounts to solving the optimization problem:
\begin{equation}\label{pb:P}\tag{P}
\min \quad \left\{\frac{\norm{v-v_0}_H^2}{2} \: : \: a \leq \nabla \cdot v \leq b,\, v // \partial \Omega\right\}.
\end{equation}
Note that, when one considers the divergence operator as an operator acting on vector fields defined on the whole $\R^d$ (extended to $0$ outside $\Omega$), the Neumann boundary condition above exactly corresponds to the fact that the divergence has no mass on $\partial\Omega$, which can be considered as included in the inequality constraints.

As a convex optimization, such a problem admits a dual problem, which we are going to use. We set
\[\psi(w) = \begin{dcases*}
0& if $a \leq w \leq b$,\\
+\infty& if not,
\end{dcases*}.\]
whose Legendre transform is
\[g(u) = \psi^\star(u) = \int bu_+ - \int a u_-,\]
so that $\psi = \psi^{\star\star} = g^\star$. Let us derive formally the dual problem by an $\inf - \sup$ exchange:
\begin{align*}
\inf_{v // \partial \Omega} \quad \left\{\frac{\norm{v-v_0}_H^2}{2} \: : \: a \leq \nabla \cdot v \leq b\right\}
&= \inf_v \frac 12 \norm{v-v_0}^2_H + \psi(\nabla \cdot v)\\
&= \inf_v \frac 12 \norm{v-v_0}^2_H + \sup_u -\langle \nabla u, v\rangle - g(u)\\
&= \inf_v \sup_u \frac 12 \norm{v-v_0}^2_H -\langle \nabla_H u, v\rangle_H - g(u)\\
&\geq \sup_u - g(u) + \inf_v \frac 12 \norm{v-v_0}^2_H -\langle \nabla_H u, v\rangle_H \\
&= - \inf_u g(u) + \sup_v \langle \nabla_H u, v\rangle_H - \frac 12 \norm{v-v_0}^2_H  \\
&= - \inf_u g(u) + \frac{\norm{\nabla_H u}_H^2}{2} + \langle \nabla_H u , v_0 \rangle_H\\
&= - \inf_u g(u) + \frac 12 \int H^{-1} \nabla u \cdot \nabla u - \int u  (\nabla \cdot v_0).
\end{align*}
Hence the dual problem reads:
\begin{equation}\label{pb:D}\tag{D}
\min_u \underbrace{\frac 12 \int H^{-1} \nabla u \cdot \nabla u - \int u (\nabla \cdot v_0)}_{f(u)} + \underbrace{\int bu_+ - \int a u_-}_{g(u)} .
\end{equation}
The $\inf - \sup$ interversion can be justified with equality via Fenchel's duality \cite[Chapter 1]{Bre} in a well-chosen Banach space. Hence there is no duality gap:
\[ \min \eqref{pb:P} + \min \eqref{pb:D} = 0.\]
As a consequence solving the dual problem provides a solution to the primal one. Indeed if $u$ is optimal for \eqref{pb:D} then $v = v_0 + \nabla_H u$ is optimal for \eqref{pb:P}. Now let us justify why it was interesting to pass by the resolution of a dual problem. Such a problem is of the form
\begin{equation}\label{pb:general}
\min_u f(u) + g(u),
\end{equation}
where $f$ is smooth, with gradient $\nabla f(u) = -\nabla \cdot (H^{-1} \nabla u) - \nabla \cdot v_0$, and $g$ is proximable:
\[\prox^\tau_g (u)(x) = \begin{dcases*}
u(x) - \tau a& if $u(x) < \tau a$,\\
0& if $\tau a \leq u(x) \leq \tau b$,\\
u(x) - \tau b& if $u(x) > \tau b$.
\end{dcases*}\]
We know how to compute the proximal operator and the gradient of $f$, since L-BFGS provides a simple method to compute the product $H^{-1} x$. Problems of the form \eqref{pb:general} with $f$ smooth (and computable gradient) and $g$ proximable can be tackled with first-order methods such as the proximal gradient method described in the previous section (also called ISTA) or a fast proximal gradient method called FISTA, introduced in \cite{BecTeb}. We opted for the latter, which is a slight modification of the proximal gradient method using an intermediary point:
\begin{equation}\tag{FISTA}\label{met:FISTA}
\left|\begin{aligned}
u_0 &\in H^1(\R^d),\\
\tilde{u}_n &= u_n + \lambda_n (u_n-u_{n-1}),\\
u_{n+1} &= \prox^\tau_g(\tilde{u}_n - \tau \nabla f (\tilde{u}_n)),
\end{aligned}\right.
\end{equation}
where $\lambda_n$ is given by some recursive formula (we refer to \cite{BecTeb} for the details). It enjoys a theoretical and pratical rate of convergence which is higher than ISTA and which is that of the classical gradient method:
\[f(u_n) - f_{opt} \leq \frac{2 L_f \abs{u_0 - u_{opt}}^2}{(n+1)^2}.\]

\subsection{Algorithms and numerical experiments}\label{sec:algo}

Following the work of \cite{OudSan}, we discretize our problem on a staggered grid : we divide the cube $Q = [-1,1]^2$ into $M^2$ subcubes of side $2/M$, the functions $U$ are defined at the center of the small cubes, while the $x$ component $V^x$ of a vector fields $V$ is defined on the vertical edges of the grid and the $y$ component $V^y$ on the horizontal edges of the grid. This is quite convenient to compute the discrete divergence of a vector field and the discrete gradient of a function.

\begin{itemize}
\item Unknowns: $(V^x_{i,j})_{\substack{1\leq i \leq M\\1\leq j \leq M+1}}$, $(V^y_{i,j})_{\substack{1\leq i \leq M+1\\1\leq j \leq M}}$, with
\[V^x_{1,j} = V^x_{M+1,j} = V^y_{i,1} = V^y_{1,M+1} = 0,\]
which means that $V$ is parallel to the boundary.
\item Objective function:
\[F(V) = \eps^{-\sigma_1} h^2 \sum_{i,j} N(\hat{V}_{i,j})^\sigma + \eps^{\sigma_2} h^2/2 \left( \sum_{i,j} \abs{\nabla_{i,j} V^x}^2 + \sum_{i,j} \abs{\nabla_{i,j} V^y}^2\right).\]
There are several definitions to give to make sense of $F$. First of all $N$ is a smooth approximation of the norm, of the form
\[N(x) = (\abs{x}^2 + \eps_s^2)^{1/2} \quad \text{ for $\eps_s$ small.}\]
The discrete vector field $\hat{V}_{i,j} = (\hat{V}^x_{i,j},\hat{V}^y_{i,j})$ is an interpolation of $(V^x,V^y)$ defined at the centers of the cubes:
\[\hat{V}^x_{i,j} = \frac{V^x_{i,j}+V^x_{i+1,j}}{2}, \quad \hat{V}^y_{i,j} = \frac{V^y_{i,j}+V^y_{i,j+1}}{2}, \quad 1 \leq i,j \leq M.\]
Finally the discrete gradient is defined as usual by
\begin{align*}
\nabla_{i,j} V^x &= ((V^x_{i,j+1}-V^x_{i,j})/h,(V^x_{i+1,j}-V^x_{i,j})/h),&
1 \leq i\leq M-1 &,1 \leq j\leq M,\\
\nabla_{i,j} V^y &= ((V^y_{i,j+1}-V^y_{i,j})/h,(V^y_{i+1,j}-V^y_{i,j})/h),&
1 \leq j\leq M-1 &,1 \leq i\leq M.\\
\end{align*}
\end{itemize}

We may now give the main algorithm and its sub-methods.

\begin{algorithm}[H]
\caption{Proximal L-BFGS for $F$}
\begin{algorithmic}
\State Data: tolerance $tol$, initial vector field $V_0$, step $\tau_0$, source $\delta$
\State $V \gets V_0$, $U \gets U_0$
\State \textbf{compute} $error$
  \While {$error > tol$}
    \State $\tau \gets \tau_0$
    \Repeat
      \State $G \gets \textsc{MultiplyBFGS}(\nabla F (V))$
      \State $V,U \gets \textsc{Project}(V - \tau G,U,\delta,\tau)$
      \State $\tau \gets \tau/2$
    \Until{$F(V)$ has decreased}
    \State \textbf{update} L-BFGS data
    \State \textbf{compute} $error$
  \EndWhile
\end{algorithmic}
\end{algorithm}
The update step for L-BFGS data consists in storing in $Y,Z,r$ the points and gradients of the $L$ previous steps, so that at step $n$:
\[Y_{L-k} = \nabla F(V_{n-k}) - \nabla F(V_{n-k-1}), \quad Z_{L-k} = V_{n-k} - V_{n-k-1}\]
for all $k = 0, \ldots, L-1$, and $r_k = 1/(Y_k \cdot Z_k)$ for all $k = 0, \ldots, L-1$.
Notice here that we do a simple backtracking line search by reducing the stepsize $\tau$ until the energy has decreased, for example until it has sufficiently decreased and satisfies the Armijo rule. Also, notice that the potential $U$ computed at step $n$ is used at the next step as initial data ; this trick extensively speeds up the computation of the projection. Finally, we took as error measurement some relative difference between two consecutive steps.

Now, as stated in Section \ref{sec:proj}, the projection on $C$ with respect to $\norm{\cdot}_H$ is computed via the FISTA method, as follows:

\begin{algorithm}[H]
\caption{Project $V_0$ on $C$ with respect to $\norm{\cdot}_H$}
\begin{algorithmic}
\State Data: tolerance $tol_p$, step $\tau_p$
\Function{Project}{$V_0,U_0,\delta,\tau$}
  \State $D_0 \gets \nabla \cdot V_0$
  \State $U \gets U_0$
  \While{$error > tol_p$}
    \State $t_p \gets t;\; t\gets (1+\sqrt{1+ 4 t_p^2})/2;\; s \gets (t_p-1)/t$
    \State $G \gets \textsc{MultiplyBFGS}(\nabla U)$
    \State $U_i \gets U + s(U-U_{old})$
    \State $U_{old} \gets U$
    \State $U \gets \textsc{Prox}(U_i - \tau_p(\nabla \cdot G-D_0),\delta,\tau)$
    \State \textbf{compute} $error$
  \EndWhile
  \State $V \gets V_0 + \textsc{MultiplyBFGS}(\nabla U)$
  \State \textbf{return} $V,U$
\EndFunction
\end{algorithmic}
\end{algorithm}

The \textsc{Prox} function is just the proximal operator associated with the discrete counterpart of $g : u \mapsto \int b u_+ - \int a u_-$ where $a = \delta - 1, b = \delta$. Thus $P = \textsc{Prox}(U,\delta,\tau)$ is defined by:
\[P_i = \begin{dcases*}
U_i - \tau (\delta_i -1)&
if $U_i < \tau(\delta_i -1)$,\\
0&
if $\tau(\delta_i -1) \leq U_i \leq \tau \delta_i$,\\
U_i - \tau \delta_i&
if $U_i > \tau\delta_i$.
\end{dcases*}\]

For the sake of completeness, we give a simple method to compute the L-BFGS multiplication $H^{-1} X$ (see \cite{LiuNoc,Noc} for details).
\begin{algorithm}[H]
\caption{L-BFGS multiplication $H^{-1} X$}
\begin{algorithmic}
\Function{MultiplyBFGS}{$X$}
  \State $G \gets X$
  \For{$ i = L, \ldots, 1$}
    \State $s_i \gets r_i Z_i \cdot G$
    \State $G \gets G - s_i Y_i$
  \EndFor
  \State $G \gets (Z_L \cdot Y) / (Y_L \cdot Y_L) \: G$
  \For{$ i = 1, \ldots, L$}
    \State $t \gets r_i Y_i \cdot G$
    \State $G \gets G + (s_i - t) Z_i $
  \EndFor
  \State \textbf{return} $G$
\EndFunction
\end{algorithmic}
\end{algorithm}

We present some numerical results obtained with $\eps_s = 10^{-4}$, on a $M \times M$ grid with $M = 201$ and $\eps = 3h$ where $h = 2/M$, the code being written in Julia. We have started with random initial values for $V$ and a smooth approximation $\delta$ of the Dirac $\delta_0$. After some days of computation on a standard laptop, one gets the following shapes and underlying networks.

With no surprise, the shape for $\alpha=0.85$ is rounder than those obtained for $\alpha=0.55$ and $\alpha=0.65$. These two are quite similar, but a simple zoom shows that the one with the smallest value of $\alpha$ is a slightly more irregular than the other. The corresponding irrigation networks are also coherent with the expected results: the branches have larger multiplicity (close to the origin) for smaller $\alpha$.

\begin{figure}[H]
    \centering
    \begin{subfigure}[t]{0.5\textwidth}
        \centering
        \includegraphics[width=\textwidth]{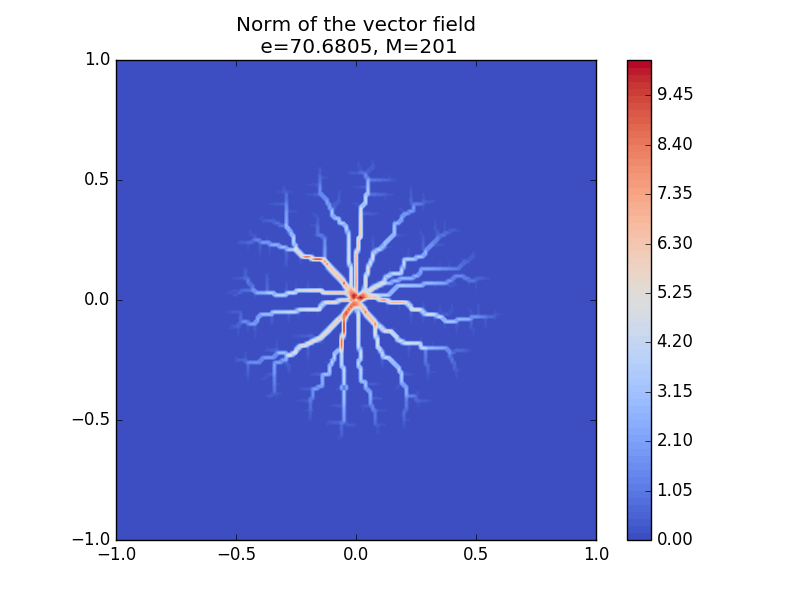}
        \caption{Norm of the vector field, $\alpha = 0.55$}
    \end{subfigure}%
    ~ 
    \begin{subfigure}[t]{0.5\textwidth}
        \centering
        \includegraphics[width=\textwidth]{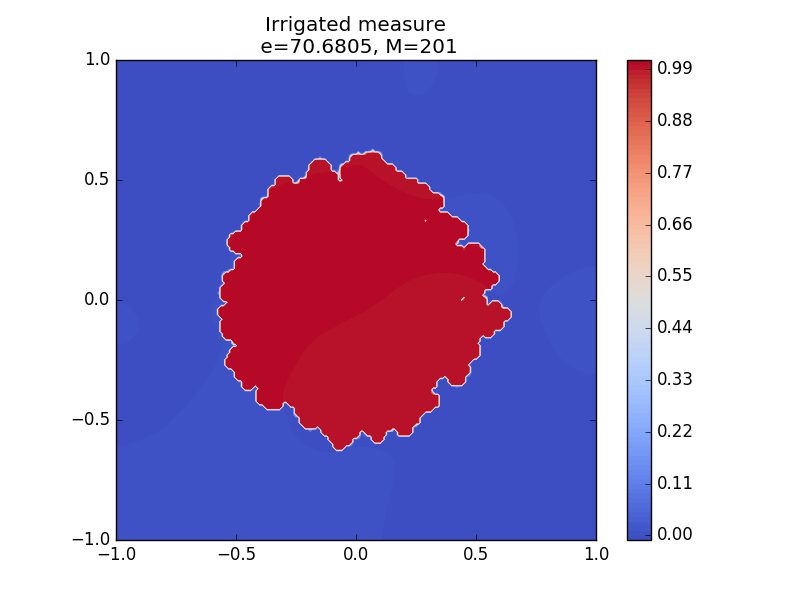}
        \caption{Irrigated measure, $\alpha = 0.55$}
    \end{subfigure}
    \\
        \begin{subfigure}[t]{0.5\textwidth}
        \centering
        \includegraphics[width=\textwidth]{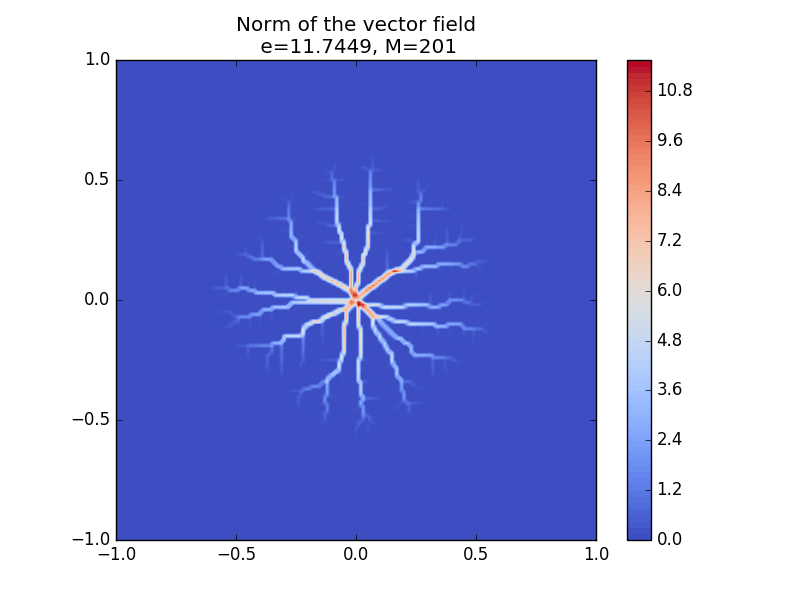}
        \caption{Norm of the vector field, $\alpha = 0.65$}
    \end{subfigure}%
    ~ 
    \begin{subfigure}[t]{0.5\textwidth}
        \centering
        \includegraphics[width=\textwidth]{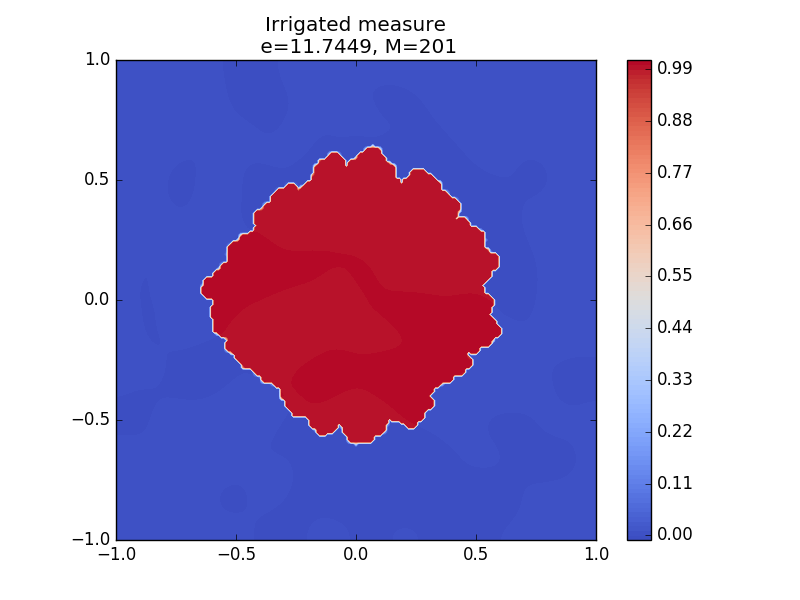}
        \caption{Irrigated measure, $\alpha = 0.65$}
    \end{subfigure}
    \\
        \begin{subfigure}[t]{0.5\textwidth}
        \centering
        \includegraphics[width=\textwidth]{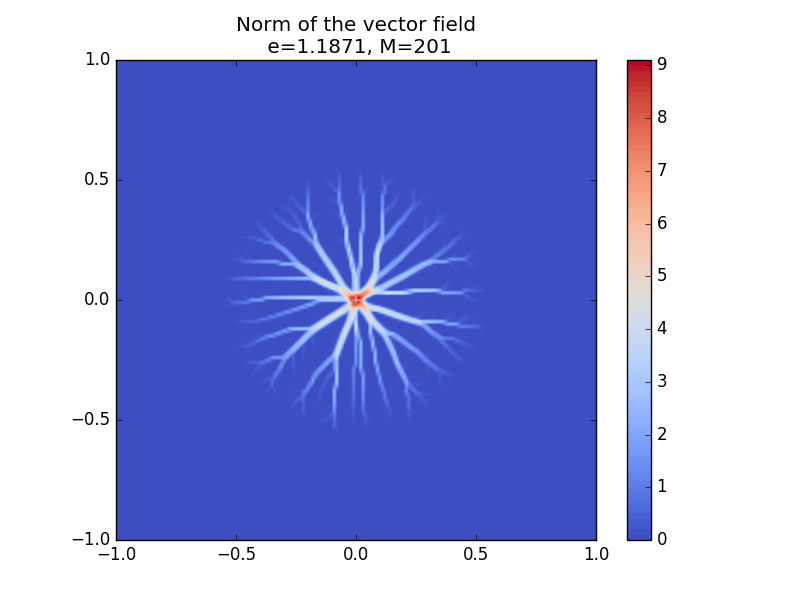}
        \caption{Norm of the vector field, $\alpha = 0.85$}
    \end{subfigure}%
    ~ 
    \begin{subfigure}[t]{0.5\textwidth}
        \centering
        \includegraphics[width=\textwidth]{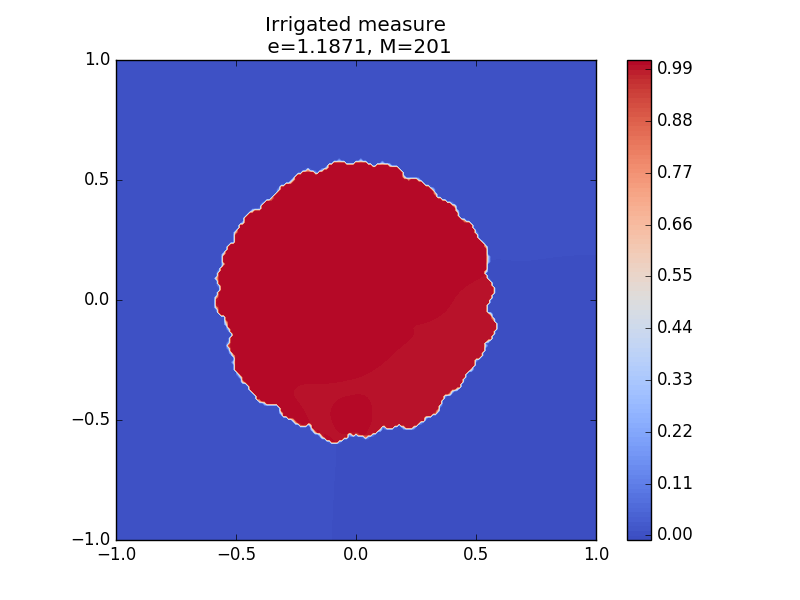}
        \caption{Irrigated measure, $\alpha = 0.85$}
    \end{subfigure}
    \caption{Algorithm output for different $\alpha$'s after $\sim 15 000$--$25 000$ iterations ($e$ stands for the computed optimal value, which is an approximation of $e_\alpha$, and $M$ for the number of discretization points on each side of the domain).}
\end{figure}

\noindent {\bf Acknowledgments and Conflicts of Interests.} The support of the ANR project ANR-12-BS01-0014-01 GEOMETRYA and of the PGMO project MACRO, of EDF and Fondation Math\'ematique Jacques Hadamard, are gratefully acknowledged. The work started during a visit of the second author to UC Davis, and also profited of a visit of the third author to Univ. Paris-Sud at Orsay, and both Mathematics Department are acknowledged for the warm hospitality.

The authors declare that no conflict of interests exists concerning this work.

\printbibliography 

\end{document}